\documentclass[12pt]{amsart}
\usepackage{amsmath, amssymb, amsfonts, amsthm, latexsym,enumerate}

\newtheorem{Theorem}{Theorem}[section]
\newtheorem{Lemma}[Theorem]{Lemma}
\newtheorem{Corollary}[Theorem]{Corollary}
\newtheorem{Proposition}[Theorem]{Proposition}

\theoremstyle{definition}
\newtheorem{Remark}[Theorem]{Remark}
\newtheorem{Example}[Theorem]{Example}

\def\To{\longrightarrow}
\def\reg{\operatorname{reg}}
\def\sdeg{\operatorname{sdeg}}
\def\depth{\operatorname{depth}}

\def\sup{\operatorname{sup}}

\def\bideg{\operatorname{bideg}}
\def\Tor{\operatorname{Tor}}

\def\a{{\bold a}}
\def\pp{{\frak p}}
\def\mm{{\frak m}}
\def\nn{{\frak n}}

\def\CC{{\mathbb C}}
\def\NN{{\mathbb N}}

\def\ZZ{{\mathbb Z}}

\def\G{{\Gamma}}

\begin{document}
\title{Saturation degree of powers of graded ideals}

\author{ Le Tuan Hoa}
\address{Institute of Mathematics, Vietnam Academy of Science and Technology, 18 Hoang Quoc Viet, 10072 Hanoi, Viet Nam}
\email{lthoa@math.ac.vn}

\author{Hop Dang Nguyen}
\address{Institute of Mathematics, Vietnam Academy of Science and Technology, 18 Hoang Quoc Viet, 10072 Hanoi, Viet Nam}
\email{ndhop@math.ac.vn}

\author{Ngo Viet Trung}
\address{Institute of Mathematics, Vietnam Academy of Science and Technology, 18 Hoang Quoc Viet, 10072 Hanoi, Viet Nam}
\email{nvtrung@math.ac.vn}


\begin{abstract}
Let $I$ be a non-nilpotent graded ideal of a standard graded algebra $R$ over a Noetherian commutative ring. 
We show that the saturation degree $\sdeg I^n$ of $I^n$ is asymptotically a constant function or a linear function whose slope is one of the generating degrees of $I$. Furthermore, we give sufficient (and necessary) conditions for a numerical functions to be the function $\sdeg I^n$ of an ideal $I$ generated by forms of the same degree. Our results may provide a better understanding of the initial behavior of the regularity function $\reg I^n$.
\end{abstract}

\dedicatory{Dedicated to the memory of Jürgen Herzog}

\maketitle


\section{Introduction} \smallskip

Let $R$ be a standard graded algebra over a Noetherian commutative ring. 
The saturation of a graded ideal $I$ in $R$ is the ideal $\widetilde I := \cup_{t \ge 0} I:\mm^t$,  
where $\mm$ denotes the graded ideal generated by forms of positive degree of $R$. 
We have  $\widetilde I_t = I_t$ for $t \gg 0$.
The {\it saturation degree} $\sdeg I$ of $I$ is the least number $t$ such that $\widetilde I$ coincides with $I$ in degree $\ge t$ with the convention $\sdeg I = -\infty$ if $\widetilde I  = I$.

Ein, H\`a and Lazarsfeld \cite{EHL} recently proved that if $I$ is a homogeneous ideal whose zero-locus defines a non-singular complex projective scheme $X$, then $\sdeg I^n$ is bounded above by a linear function whose slope is the maximal generating degree $d(I)$ of $I$. 
This reminds us of an earlier result of Bertram, Ein, and Lazarsfeld \cite{BEL} that the Castelnuovo-Mumford regularity of $\Im^n$, where $\Im$ is the ideal sheaf of $X$, is bounded above by the same linear function.
That result had initiated the study on the Castelnuovo-Mumford regularity $\reg I^n$, which culminated in the works of Cutkosky, Herzog, Trung \cite{CHT} and Kodiyalam \cite{Ko}, who independently found out that if $I$ is an arbitrary homogeneous ideal of a polynomial ring $R$, then $\reg I^n$ is an asymptotically linear function whose slope is $\le d(I)$. Their result was later extended to the case $R$ is a standard graded algebra by Trung and Wang \cite{TW}. One may raise the question whether the same phenomenon holds for the function $\sdeg I^n$.

We will give a positive answer to the above question by the following result.
We call a graded ideal $J \subseteq I$ a {\em reduction} of $I$ if 
$I^{n+1} = JI^n$ for $n \gg 0$. Note that $I$ is itself a reduction of $I$.
\medskip

\noindent {\bf Theorem \ref{sdeg}.}
{\em $\sdeg I^n$ is asymptotically a constant function or a linear function whose slope is one of the generating degrees of any reduction of $I$.}
\medskip

Theorem \ref{sdeg} will be deduced from a result on partial regularities.

For a finitely generated graded $R$-module $H$ we set $a(H) := \sup\{n|\ H_n \neq 0\} \in \ZZ\cup \{\infty, -\infty \}$ with the convention $a(H) = -\infty$ if $H = 0$. Let $M$ be a finitely generated graded $R$-module.
For any given number $t \le \dim M$ we define
$$\reg_t M := \max\{a(H_\mm^i(M))+i|\ i \le t\},$$
where $H_\mm^i(M)$ denotes the $i$th local cohomology module of $M$ with respect to $\mm$.  
We call $\reg_t M$ the $t$-th {\em partial regularity} of $M$. This invariant has been studied in \cite{Sch,Tr1,Tr2,Tr3}.
Note that $\reg M := \reg_{\dim M} M$ is the Castelnuovo-Mumford regularity of $M$.

It is a less known fact that the partial regularities are asymptotically the constant function $-\infty$ or a linear function whose slope is one of the generating degrees of any reduction of $I$.
We will see that $\sdeg I^n = \reg_1 I^n$ for $n \gg 0$ if $H_\mm^1(I^n) \neq 0$ for $n \gg 0$ and that $\sdeg I^n = a(H_\mm^0(R))$ for $n \gg 0$ if $H_\mm^1(I^n) = 0$ for $n \gg 0$. Combining these facts, we obtain Theorem \ref{sdeg}.

Ein, H\`a and Lazarsfeld \cite[p. 1532]{EHL} asked whether their bound for the function $\sdeg I^n$ is best possible. Inspired by Theorem \ref{sdeg} we will present a homogeneous ideal $I$ whose zero-locus defines a non-singular complex projective scheme such that $\sdeg I^n$ is a linear function whose slope and intercept are much less than those of their bound. 

Another motivation for this paper comes from the works of Eisenbud and Harris \cite{EH}, 
Eisenbud and Ulrich \cite{EU}.
Let $I$ be an {\em equigenerated} ideal, i.e. $I$ is generated by forms of the same degree. 
They showed that the initial behavior of the function $\reg I^n$ may have some common properties before it becomes a linear function. 
Inspired by their results, we will study the initial behavior of the function $\sdeg I^n$.
By Theorem \ref{sdeg}, $\sdeg I^n$ is asymptotically a constant function or a linear function with slope $d$, where $d$ is the degree of the generators of $I$.
We will concentrate in the case $\sdeg I^n$ is asymptotically a linear function with slope $d$. 
In this case, we call the sequence $\{\sdeg I^n - dn\}_{n \ge 1}$, the {\em defect sequence} of the function $\sdeg I^n$.  

We will see that this sequence is convergent to a non-negative number though 
this sequence may have {\em initial negative values}. 
This is a contrast to the function $\reg I^n$, where we always have $\reg I^n - dn \ge 0$. 

If $\dim R/I = 0$, $\sdeg I^n = \reg R/I^n + 1$. The function $\reg R/I^n$ was studied in \cite{EH,HNT}.
Using the results there, we obtain the following complete characterization of the defect sequence of the function $\sdeg I^n$ for the case $\dim R/I = 0$.
\medskip

\noindent {\bf Theorem \ref{defect1}.}
{\em A sequence of integers $\{e_n\}$  is the defect sequence of the function $\sdeg I^n$ of an ideal $I$ equigenerated in degree $d$ with $\dim R/I = 0$ if and only if it is a weakly decreasing sequence of non-negative integers and $e_n \le e_{n+1} + d$ for all $n \ge 1$.} 
\medskip

The condition $e_n \le e_{n+1} + d$ for all $n \ge 1$ means that the function $\sdeg I^n$ is weakly increasing.

If $\dim R/I > 0$, we show that the condition $e_n \le e_{n+1} + d$ for all $n \ge 1$ is the only obstruction for a convergent sequence of non-negative integers to be the defect sequence of a weakly increasing function $\sdeg I^n$.
\medskip

\noindent {\bf Theorem \ref{defect2}.}
{\em Let $\{e_n\}$ be an arbitrary convergent sequence of non-negative integers. Then $\{e_n\}$ is the defect sequence of a weakly increasing function $\sdeg I^n$ of an ideal $I$ equigenerated in degree $d$ with $\dim R/I > 0$ if and only if $e_n \le e_{n+1} + d$ for all $n \ge 1$.}
\medskip

As a consequence, any convergent sequence of non-negative integers $\{e_n\}$ is the defect sequence of the function $\sdeg I^n$ for some equigenerated ideal $I$. In fact, the condition $e_n \le e_{n+1} + d$ is satisfied if $d$ is large enough. If $d$ is given, we can only show that any weakly decreasing sequence of non-negative integers
is a defect sequence of the function $\sdeg I^n$.
\medskip

\noindent {\bf Theorem \ref{defect3}.}
{\em Let $\{e_n\}$ be an arbitrary weakly decreasing sequence of non-negative integers. Then $\{e_n\}$ is the defect sequence of the function $\sdeg I^n$ of an ideal $I$ equigenerated in any degree $d$.}
\medskip

In particular, Theorem \ref{defect2} and  Theorem \ref{defect3} show that the initial sequence of the function $\sdeg I^n$ can be any finite {\em monotone sequence} of non-negative integers before it becomes a linear function. 
Furthermore, this initial sequence can have local maxima and minima at any sequence of prescribed places.

Throughout this paper, $R$ is a standard graded algebra over a Noetherian commutative ring $A$ and $I$ is a non-nilpotent proper graded ideal of $R$ (we need $I^n \neq 0$ for all $n$).
The asymptotic behavior of the functions $\reg_t I^n$ and $\sdeg I^n$ will be studied in Section 2 and Section 3, respectively. In Section 4 we study the initial behavior of the function $\sdeg I^n$. 
\medskip

\noindent{\em Acknowledgements}.
The content of this paper is an improved version of the last part of an unpublished version of the manuscript \cite{HNT}. We are grateful to Aldo Conca and the anonymous referee for thoughtful commments that have helped us to improve this paper.
The first author is supported by NAFOSTED grant 101.04-2024.07. The second and last authors are supported by NAFOSTED grant 101.04-2023.30. The second author is also supported by grant CSCL01.01/25-26 of the Institute of Mathematics, VAST. 


\section{Preliminaries}

In this section we will see that $\reg_t I^n$ is asymptotically a linear function.
This fact are more or less a direct consequence of results in \cite{BC, BCRV, CHT, Sch, Tr1, Tr2, TW}.
To see this we need some explanation.
We will follow the approach of \cite{TW}, which proved that $\reg I^n$ is asymptotically a linear function in the general setting that $R$ is a standard graded algebra over a Noetherian commutative ring $A$.

Let $J$ be a graded ideal of $R$. 
Suppose that $J = (f_1,...,f_s)$, where $f_1,...,f_s$ are graded elements of $R$ of degree $d_1,...,d_s$, respectively. 
Let $R[Jt] := \oplus_{n \ge 0}J^n t^n$ be the Rees algebra of $J$.
Then we can represent $R[Jt]$ as a quotient ring of a bigraded polynomial ring $S := R[y_1, . . . ,y_s]$, where $\bideg f = (\deg f, 0)$ for $f \in R$ and $\bideg y_i = (d_i , 1)$, $i = 1, . . . ,s$.  
Assume that $J$ is a reduction of $I$.
Then there exists a number $r$ such that $I^{n+r} = J^nI^r$ for $n \ge 1$.
From this it follows that $R[It]$ is a finitely generated bigraded $S$-module.

Let $E$ be an arbitrary finitely generated bigraded $S$-module. For $n \in \NN$ let 
$$E_n := \oplus_{a \in \ZZ}E_{(a,n)}.$$
Then $E_n$ is a  finitely generated graded $R$-module. 
For instance, $E_n \cong I^n$ if $E = R[It]$.
We will show that $\reg_t E_n$ is asymptotically a linear function. 
This will follow from the following results.  

\begin{Proposition} \label{a} \cite[Theorem 8.3.4(a)]{BCRV} {\rm (cf. \cite[Theorem 4.3]{BCH},  \cite[Lemma 2.1]{BC}, \cite[Proposition 2.1]{TW})}
Assume that $R = A$. Then $a(E_n)$ is asymptotically the constant function $-\infty$ or a linear function with slope in $\{d_1,..., d_s\}$.
\end{Proposition}

Let $M$ be a finitely generated graded $R$-module. We call a sequence $z_1,...,z_t$ of graded elements in $R$ a {\em filter-regular sequence} for $M$ if $z_i \not\in \pp$ for any associated prime $\pp \not\supseteq \mm$ of $M/(z_1,...,z_{i-1})M$, $i = 1,...,t$. For convenience, we set $(z_1,...,z_0) = 0$. Note that filter-regular sequences of any length exist. It may happen that $z_i \in (z_1,...,z_{i-1})$.

\begin{Proposition} \label{filter} \cite[Proposition 2.2(ii)]{Tr2} 
Let $z_1,...,z_t$ be a filter-regular sequence of linear forms for $M$. Then
$$\reg_t M = \max\big\{a((z_1,...,z_i)M:\mm/(z_1,...,z_i)M)|\ i = 0,...,t\big\}.$$
\end{Proposition}

This characterization was originally proved for $\reg_t R$ but the proof can be extended to the module case in the same way. It is a modified version of an earlier characterization for $\reg_t R$ in terms of a filter-regular sequence of $t+1$ linear forms \cite[Proposition 2.2]{Tr1}.

\begin{Theorem} \label{partial1}
Let $E$ be an arbitrary finitely generated bigraded $S$-module.
Then $\reg_t E_n$ is asymptotically the constant function $-\infty$ or a linear function with slope in $\{d_1,..., d_s\}$. 
\end{Theorem}

\begin{proof} 
Assume that $R$ is generated as an $A$-algebra by the linear forms $x_1,\ldots,x_t$. 
Let $U=(u_{ij})_{t\times t}$ be a matrix of indeterminates, and set $z_i=\sum_{j=1}^t u_{ij}x_j$ for $i=1,\ldots,t$. Let $A'=A[U,1/\det(U)], R'=R\otimes_A A', S'=S\otimes_A A'$, and  $E'=E\otimes_A A'$. 
Then $R'$ is a standard graded $A'$-algebra, $E'$ is a finitely generated bigraded $S'$-module, and the flatness of the map $A\to A'$ implies that $\reg_t E'_n=\reg_t E_n$ for all $t$ and $n$.
As in the proof of \cite[Lemma 1.2]{TW}, $z_1,\ldots,z_t$ is an $E'_n$-filter-regular sequence of linear form in $R'$ for every $n$. 

Replacing $R$ by $R'$, we may assume that there exists a filter-regular sequence $ z_1,...,z_t$ of linear forms for all $E_n$, $n \ge 1$.
Let $E^i = (z_1,\ldots,z_i)E:\mm/(z_1,\ldots,z_i)E$ for $i = 0,...,t.$ Then $E^i$ is a finitely generated bigraded module over the bigraded ring $S/\mm S = A[y_1,...,y_s]$. By Proposition \ref{a},   $a((E^i)_n) = -\infty$ or $a((E^i)_n)$ is a linear function with slope in $\{d_1,...,d_s\}$ for $n \gg 0$. By Proposition \ref{filter}, $\reg_t E_n = \max\big\{a((E^i)_n)|\ i = 0,\ldots,t\big\},$ which implies the conclusion. 
\end{proof}

If $E = R[It]$, then $E_n \cong I^n$ for all $n$. 
Hence from Theorem \ref{partial1}  we immediately obtain the following result on the function $\reg_t I^n$.

\begin{Corollary} \label{partial2}
$\reg_t I^n$ is asymptotically the constant function $-\infty$ or a linear function whose slope is one of the generating degrees of any reduction of $I$.
\end{Corollary}

\begin{Remark} 
If $R$ is a polynomial ring over a field $k$, we know by \cite[Theorem 3.11]{Sch} that 
$$\reg_t M = \max\{a(\Tor_i^R(k,M))-i|\ \dim R - t \le i \le \dim R\}.$$
The right side is called the $(\dim R-t)$-regularity in \cite{BCP}. By \cite[Theorem 3.1]{CHT}, $a(\Tor_i^R(k,I^n))$ is asymptotically a linear function. Using these results, one can also prove Corollary \ref{partial2} in the case $R$ is a standard graded algebra over a field.
\end{Remark}

\begin{Remark} 
Suppose that $\reg I^n = dn + e$ for $n \gg 0$. It was shown in \cite[Theorem 3.2]{TW} that
$d = \min\{d(J)|\ \text{$J$ is a reduction of $I$}\}.$
Using \cite[Example 8.4.8 or 8.4.9]{BCRV} one can shows that $d$ needs not be the slope of the asymptotic function $\reg_t I^n$ for $t < \dim I$. 
\end{Remark}


\section{Asymptotic behavior of $\sdeg I^n$}

In this section we study the behavior of the function $\sdeg I^n$ for $n \gg 0$. 
By definition, we have
$$\sdeg I^n = a(\widetilde{I^n}/I^n) + 1 = a(H_\mm^0(R/I^n)) + 1.$$

From the short exact sequence $0 \To I^n \To R \To R/I^n \To 0$ we get the following long exact sequence of local cohomology modules
\[\tag{$\star$}
 0 \To H_\mm^0(I^n) \To H_\mm^0(R) \To H_\mm^0(R/I^n) \To H_\mm^1(I^n) \To H_\mm^1(R).
 \]
Hence, $a(H_\mm^0(R/I^n))$ can be estimated by the $a$-invariants of the other local cohomology modules of this sequence.

\begin{Lemma} \label{I^n} ~\par
{\rm (1)} $H_\mm^0(I^n) = 0$ for $n \gg 0$. \par
{\rm (2)} $\reg_1 I^n = a(H_\mm^1(I^n))+1$ for $n \gg 0$. \par
{\rm (3)} $a(H_\mm^1(I^n))$ is asymptotically the constant function $- \infty$ or a linear function with positive slope. 
\end{Lemma}

\begin{proof} 
(1) $H_\mm^0(I^n) = H_\mm^0(R)  \cap I^n = 0$ for $n > a(H_\mm^0(R))$.  
 \par
(2) For $n > a(H_\mm^0(R))$, $a(H_\mm^0(I^n)) = -\infty$ for $n \gg 0$, whence
$$\reg_1 I^n = \max\{a(H_\mm^0(I^n)), a(H_\mm^1(I^n))+1\} = a(H_\mm^1(I^n))+1.$$ \par
(3)  By Corollary \ref{partial2}, $\reg_1 I^n$ is asymptotically the constant function $- \infty$ or a linear function with positive slope. Hence, so is $a(H_\mm^1(I^n))$. 
\end{proof}

By Lemma \ref{I^n}(3), either $H_\mm^1(I^n) = 0$ for $n \gg 0$ or $H_\mm^1(I^n) \neq 0$ for $n \gg 0$. 
Depending on these conditions, we have the following asymptotic formulas for $\sdeg I^n$.

\begin{Proposition} \label{cases}
~ \par
{\rm (1)} If $H_\mm^1(I^n) = 0$ for $n \gg 0$, then $\sdeg I^n = a(H_\mm^0(R))+1$ for $n \gg 0$. \par
{\rm (2)} If $H_\mm^1(I^n) \neq 0$ for $n \gg 0$, then $\sdeg I^n = \reg_1 I^n$ for $n \gg 0$.
\end{Proposition}

\begin{proof} 
(1) By Lemma \ref{I^n}(1), $H_\mm^0(I^n) = 0$ for $n \gg 0$.
If $H_\mm^1(I^n) = 0$ for $n \gg 0$, then $H_\mm^0(R) \cong H_\mm^0(R/I^n)$ for $n \gg 0$ by $(\star)$. Therefore, 
$$\sdeg I^n = a(H_\mm^0(R/I^n))+1 = a(H_\mm^0(R))+1$$ 
for $n \gg 0$.

(2) If $H_\mm^1(I^n) \neq 0$ for $n \gg 0$, then $a(H_\mm^1(I^n)) \neq -\infty$ for $n \gg 0$. 
By Lemma \ref{I^n}(3), $a(H_\mm^1(I^n))$ is asymptotically a linear function of positive slope.  
From this it follows that
$$a(H_\mm^1(I^n)) > \max\{a(H_\mm^0(R)),a(H_\mm^1(R))\}$$
for $n \gg 0$. 
Using $(\star)$ we deduce that $a(H_\mm^0(R/I^n)) = a(H_\mm^1(I^n))$ for $n \gg 0$. So we obtain
$$\sdeg I^n = a(H_\mm^0(R/I^n)) +1= a(H_\mm^1(I^n))+1 = \reg_1 I^n$$
for $n \gg 0$ by Lemma \ref{I^n}(2).
\end{proof}

By Corollary \ref{partial2}, $\reg_1 I^n$ is asymptotically the constant function $-\infty$ or a linear function whose slope is one of the generating degrees of any reduction of $I$. Combining this fact with Proposition \ref{cases} we immediately obtain the following result.

\begin{Theorem} \label{sdeg}
$\sdeg I^n$ is asymptotically a constant function or a linear function whose slope is one of the generating degrees of any reduction of $I$.
\end{Theorem} 

The condition $H_\mm^1(I^n) = 0$ for $n \gg 0$ can be easily checked without computing the local cohomology module $H_\mm^1(I^n)$.

\begin{Lemma} \label{U}
Let $U = \cup_{t \ge 0}0:\mm^t$. Then
$H_\mm^1(I^n) = 0$ for $n \gg 0$ if and only if $\widetilde{I^n} = U+I^n$ for $n \gg 0$.
\end{Lemma}

\begin{proof} 
Note that $H_\mm^0(R) = U$ and $H_\mm^0(R/I^n) = \widetilde{I^n}/I^n$.
Then $U +I^n /I^n$ is the image of $H_\mm^0(R)$ in $H_\mm^0(R/I^n)$.
Hence $\widetilde{I^n} = U+I^n$ if and only if the map $H_\mm^0(R) \to H_\mm^0(R/I^n)$ is surjective.

If $H_\mm^1(I^n) = 0$, then $H_\mm^0(R) \to H_\mm^0(R/I^n)$  is surjective by $(\star)$. 
Conversely, if $H_\mm^0(R) \to H_\mm^0(R/I^n)$  is surjective for $n \gg 0$, then $H_\mm^1(I^n) \to H_\mm^1(R)$ is injective for $n \gg 0$. This implies $a(H_\mm^1(I^n)) \le a(H_\mm^1(R))$ for $n \gg 0$. 
Hence $a(H_\mm^1(I^n))$ cannot be a linear function with positive slope for $n \gg 0$.
By Lemma \ref{I^n}(3), $a(H_\mm^1(I^n)) = -\infty$ for $n \gg 0$, which means $H_\mm^1(I^n) = 0$ for $n \gg 0$.
\end{proof}

\begin{Example} \label{constant}
Let $R = k[x,y,z]/(x^2,xy,xz)$ and $I = zR$. Then $U = xR$, $I^n = z^nR$ and $\widetilde{I^n} = (x,z^n)R$ for all $n \ge 1$.  Hence $H_\mm^1(I^n) = 0$ for all $n \ge 1$.
We have $\sdeg I^n = 2$ for all $n \ge 1$.
\end{Example}

Ein, H\`a and Lazarsfeld \cite[Theorem A]{EHL} proved that if $R = \CC[x_0,...,x_r]$ is a polynomial ring over the complex numbers and $I = (f_0,...,f_p)$ an ideal generated by forms of degrees $d_0 \ge \cdots \ge d_p$ such that the projective scheme cut out by $f_0,...,f_p$ is nonsingular, then
$$\sdeg I^n \le d_0 n + d_1 + \cdots + d_r - r$$
for all $n \ge 1$, where $d_{p+1} = \cdots = d_r = 0$ if $p < r$. 
They asked whether this bound is best possible \cite[p.~1532]{EHL}. 
Theorem \ref{sdeg} suggests that the slope $d_0$ is not optimal. 
In fact, the following example gives such an ideal $I$ for which $\sdeg I^n$ is a linear function for all $n \ge 1$ with slope and intercept arbitrarily less than that of this bound.

\begin{Example} \label{EHL}
Let $R = \CC[x_0,...,x_r]$, $r \ge 1$, and $I = z(x_0^2,...,x_r^2,x_0 \cdots x_r)$,  where $z$ is an arbitrary linear form. Then the projective scheme cut out by the generators of $I$ is nonsingular. By the above result of Ein, H\`a and Lazarsfeld we have 
$$\sdeg I^n \le (r+2)n + 2r \text{ for all } n \ge 1.$$ 
We shall see that this bound is far from being optimal for all $n$. \par

We have $\widetilde{I^n} = (z^n)$. Let $Q = (x_0^2,...,x_r^2,x_0 \cdots x_r)$. Then
$$\widetilde {I^n}/I^n = (z^n)/z^nQ^n \cong (R/Q^n)(-n).$$ 
From this it follows that $\sdeg I^n = a(R/Q^n)+n+1.$
Let $f$ be a monomial which has maximal degree among the monomials not contained in $Q^n$.
 We can always write $f = gh$, where $g$ is a product of the monomials $x_0^2,...,x_r^2$ and $h$ is a squarefree monomial.
 Since $fx_i \in Q^n$ for all $i = 0,...,r$, we must have $g \in (x_0^2,...,x_r^2)^{n-1} \subset Q^{n-1}.$ 
 Since $f \not\in Q^n$, $\deg g = 2(n-1)$ and $h \not\in Q$. 
 Since $h$ has maximal degree among the squarefree monomials not contained in $Q$, $h$ is a product of $r$ variables. Hence $\deg f = 2(n-1) + r$. Therefore, $a(R/Q^n) = 2(n-1)+r$. So we get
 $$\sdeg I^n = 3n + r -1 \text{ for all } n \ge 1. $$
 Obviously, $((r+2)n + 2r) - (3n + r -1) = (r-1)(n + 1)$ can be arbitrarily large.
\end{Example}

It would be of interest to find an example analogous to Example 3.6 where the ideal $I$ is saturated.

\section{Initial behavior of $\sdeg I^n$}

In this section we study the behavior of the function $\sdeg I^n$ before it becomes a linear function.

{\em From now on, we assume that $I$ is a graded ideal equigenerated in degree $d$}.
By Theorem \ref{sdeg}, $\sdeg I^n$ is a linear function of slope $d$ for $n \gg 0$ if it is not a constant function for $n \gg 0$. 

If $\sdeg I^n$ is a linear function of slope $d$ for $n \gg 0$, 
we call the sequence $\{\sdeg I^n -dn\}$, $n \ge 1$, the {\em defect sequence} of the function $\sdeg I^n$. 
The terminology is adapted from \cite{Be}. This sequence is convergent to a non-negative number.

\begin{Proposition} \label{intercept}
Suppose that $\sdeg I^n = nd +e$ for $n \gg 0$. Then $e \ge 0$.
\end{Proposition}

\begin{proof}
By Lemma \ref{cases}(2), $\sdeg I^n = \reg_1 I^n$ for $n \gg 0$.
As in the proof of Theorem \ref{partial1}, we may assume that there exists a filter-regular linear form $z$ for all $I^n$. By Proposition \ref{filter}, 
\begin{align*}
\reg_1 I^n & = \max\big\{a(0:_{I^n}\mm),a(zI^n:_{I^n}\mm/zI^n)\big\}\\
& = \max\big\{a((0:\mm)\cap I^n), a((zI^n : \mm)\cap I^n/zI^n)\big\}.
\end{align*}
We have $(0:\mm)\cap I^n = 0$ for $n \gg 0$. Hence, $(zI^n : \mm)\cap I^n \neq 0$ for $n \gg 0$.
Since $I^n$ is generated by forms of degree $dn$, $a((zI^n : \mm)\cap I^n) \ge dn$ for $n \gg 0$.
Therefore, $\reg_1 I^n \ge dn$ for $n \gg 0$, which implies $e \ge 0$.
\end{proof}

The following example shows that the defect sequence of $\sdeg I^n$ can be a weakly decreasing sequence of {\em negative integers} before it becomes a constant. It is constructed by using the theory of edge ideals.

\begin{Example} \label{negative}
Let $S = k[x_1,...,x_{2t+1}]$, $t \ge 2$, and $P = (x_1x_2,...,x_{2t}x_{2t+1},x_{2t+1}x_1)$,
the edge ideal of $(2t+1)$-cycle. Let $S^* = S[x_0]$ and $Q = x_0(x_0,...,x_{2t+1})$. Set 
$R = S^*/Q$ and $I = (P,Q)/Q.$ Note that as an ideal of $R$, $I$ is minimally generated by elements of degree $2$.

We claim that $\sdeg I^n = 2$ for $n \le t$ and  $\sdeg I^n = 2n$ for $n > t$.
In particular, we have $\sdeg I^n - 2n = 2-2n \le 0$ for $n = 2,...,t$. 

Since $\widetilde Q = (x_0)$, we have
$\widetilde{(P^n,Q)} = \widetilde{(P^n,x_0)} = (H_n,x_0)$, where $H_n$ denotes the saturation of $P^n$ in $S$.
Hence, $\widetilde{I^n}/I^n = (H_n,x_0)/(P^n,Q)$.
Consider the exact sequence
$$0 \to  (P^n,x_0)/(P^n,Q) \to \widetilde{I^n}/I^n \to (H_n,x_0)/(P^n,x_0) \to 0.$$
Note that $(P^n,x_0)/(P^n,Q) = (R/\mm)(-1)$ and $(H_n,x_0)/(P^n,x_0) \cong H_n/P^n$.
Then
$$a(\widetilde{I^n}/I^n) = \max\{1, a(H_n/P^n)\}.$$

For $n \le t$, we know that $\depth S/P^n > 0$ \cite[Lemma 3.10]{RTY}.
Thus,  $H_n = P^n$. Therefore, $a(\widetilde{I^n}/I^n) = 1$, which implies $\sdeg I^n = 2$. 
 
For $n > t$, it suffices to show that $a(H_n/P^n) = 2n-1$.
We note that $a(H_n/P^n) = a(P^n:\nn/P^n)$, where $\nn$ denotes the maximal homogeneous ideal of $S$.
Let $x^\a$ be an arbitrary monomial of $P^n:\nn \setminus P^n$, where $\a$ denotes the exponent vector of the monomial. 
If $x^\a \in P^{n-1}: \nn\setminus P^{n-1}$, we may assume that $\deg x^\a \le 2(n-1)-1$.
If $x^\a \not \in P^{n-1}: \nn\setminus P^{n-1}$, then the associated weighted graph $\G_\a$ (see definition in \cite[Section 1]{LT}) is matching critical by \cite[Lemma 1.6 and Definition 2.1]{LT}. This implies $\deg x^a = 2(n-1)+1 \le 2n -1$ \cite[Corollary 2.7]{LT}. Therefore, $a(P^n:\nn/P^n) \le 2n-1$.
On the other hand, for $f = x_1 \cdots x_{2t+1}$, we have $x_if \in P^{t+1}$ for $i = 1,...,2t+1$.
This implies $(x_1x_2)^{n-t-1}f \in P^n :\nn$. 
Since $P^n$ is generated by elements of degree $2n$ and $\deg (x_1x_2)^{n-t-1}f = 2n-1$, we have $(x_1x_2)^{n-t-1}f \not\in P^n$. Therefore, $a(P^n:\nn/P^n) \ge 2n-1$. So we get $a(P^n:\nn/P^n) = 2n-1$, which implies $\sdeg I^n = 2n$.
\end{Example}

In general, $\sdeg I^n = a(\widetilde{I^n}/I) + 1 \le \reg R/I^n + 1$, which becomes an equality if $\dim R/I = 0$. 
It is known that $\reg R/I^n$ is asymptotically a linear function of slope $d$ with $\reg R/I^n - dn +1 \ge 0$. The sequence $\{\reg R/I^n - dn +1\}$, $n \ge 1$, is called the defect sequence of the function $\reg R/I^n$. 

If $\dim R/I = 0$, there is a complete characterization of sequences of non-negative integers which are the defect sequences of $\reg R/I^n$  \cite[Theorem 4.9]{HNT} (the necessary part is due to \cite[Proposition 1.1]{EH}).
Translating this characterization for the defect sequence of $\sdeg I^n$ we immediately obtain the following result. 

\begin{Theorem} \label{defect1}
A sequence of integers $\{e_n\}$ is the defect sequence of the function $\sdeg I^n$ of an ideal $I$ equigenerated in degree $d$ with $\dim R/I = 0$ if and only if it is a weakly decreasing sequence of non-negative integers and $e_n \le e_{n+1} + d$ for all $n \ge 1$.
\end{Theorem}

Note that the condition $e_n \le e_{n+1} + d$ for all $n \ge 1$ means that $\sdeg I^n$ is a weakly increasing function.

If $\dim R/I > 0$, we know by Example \ref{negative} that the the defect sequence of $\sdeg I^n$ may have negative values. Therefore, the initial behavior of the defect sequence of $\sdeg I^n$ is not the same as that of $\reg R/I^n$. 

We can show that the condition $e_n \le e_{n+1} + d$ for all $n \ge 1$ is the only obstruction for a convergent sequence of non-negative integers to be the defect sequence of a weakly increasing function $\sdeg I^n$.

\begin{Theorem} \label{defect2}
Let $\{e_n\}$ be an arbitrary convergent sequence of non-negative integers. Then $\{e_n\}$ is the defect sequence of a weakly increasing function $\sdeg I^n$ of an ideal $I$ equigenerated in degree $d$ with $\dim R/I > 0$ if and only if $e_n \le e_{n+1} + d$ for all $n \ge 1$.
\end{Theorem}

\begin{proof}
We only need to construct an ideal $I$ equigenerated in degree $d$ with $\dim R/I > 0$ such that $\sdeg I^n = nd + e_n$ for all $n \ge 1$ if $e_n \le e_{n+1} + d$ for all $n \ge 1$.

Let $m$ be the stability index of the sequence $\{e_n\}$ (which is the least integer $m$ such that $e_n = e_m$ for all $n > m$). Let $S = k[x_1,...,x_m,y]$, $P = (x_2,...,x_m)$, and
$$Q = (x_1y^d,x_2^{d}y^{e_2+1},...,x_m^{(m-1)d}y^{e_m+1},x_1^{e_1+1},x_2^{2d},...,x_{m-1}^{(m-1)d},D),$$
where $D = (x_ix_j|\ i,j = 1,...,m, i \neq j)$. The set $\{x_2^{2d},...,x_{m-1}^{(m-1)d}\}$ is empty if $m \le 2$.
Let $R = S/Q$ and $I = (P^d+Q)/Q$.  
The inclusion of $D$ in $Q$ allows us to deal only with monomials of the form $x_i^ay^b$, $i = 1,...,m$.

Note that as an ideal of $R$, $I$ is minimally generated by elements of degree $d$.

We have
$\widetilde{I^n}/I^n = \widetilde{(P^{dn},Q)}/(P^{dn},Q).$ Hence
$$\sdeg I^n = a(\widetilde{(P^{dn},Q)}/(P^{dn},Q))+1.$$
Note that $a(\widetilde{(P^{dn},Q)}/(P^{dn},Q))$ is the maximal degree of the monomials in $\widetilde{(P^{dn},Q)} \setminus (P^{dn},Q)$.
Since $y$ is a system of parameters for $S/(P^{dn},Q)$, 
$$
\widetilde{(P^{dn},Q)} = \bigcup_{t \ge 0}(P^{dn},Q):y^t   
                       =(x_1,x_2^d,\ldots,x_m^{(m-1)d},x_2^{nd},\ldots,x_m^{nd},D).
$$

For $n < m$, we have
\begin{align*} 
\widetilde{(P^{dn},Q)} & = (x_1,x_2^d,...,x_n^{(n - 1)d},x_{n+1}^{nd},...,x_m^{nd},D),\\
(P^{dn},Q) & = (x_1y^{d},x_2^dy^{e_2+1}...,x_n^{(n-1)d}y^{e_n+1},
x_1^{e_1+1},x_2^{2d},...,x_n^{nd},x_{n+1}^{nd},...,x_m^{nd},D).
\end{align*} 
The monomials in $\widetilde{(P^{dn},Q)} \setminus (P^{dn},Q)$ are monomials in the sets $(x_1) \setminus (x_1y^{d},x_1^{e_1+1})$, $(x_i^{(i-1)d}) \setminus (x_i^{(i-1)d}y^{e_i+1},x_i^{id})$, $i = 2,...,n$.
From this it follows that $a(\widetilde{(P^{d},Q)}/(P^{d},Q))$ is attained among the degrees of the elements 
$x_1^{e_1}y^{d-1},x_{i}^{id-1}y^{e_{i}},$ $i = 2,...,n$.
Hence,
$$a(\widetilde{(P^{dn},Q)}/(P^{dn},Q)) = \max\{id + e_i-1|\ i = 1,...,n\} = nd + e_n-1$$
because $id + e_i \le nd+e_n$ for $i = 1,...,n$. Therefore, $\sdeg I^n = nd + e_n$.

For $n \ge m$, we have
\begin{align*} 
\widetilde{(P^{dn},Q)} & = (x_1,x_2^d,...,x_{m}^{(m - 1)d},D)\\
(P^{dn},Q) & = (x_1y^{d},x_2^dy^{e_2+1}...,x_m^{(m-1)d}y^{e_m+1},
x_1^{e_1+1},x_2^{2d},...,x_{m-1}^{(m-1)d},x_m^{nd},D).
\end{align*} 
The monomials in $\widetilde{(P^{dn},Q)} \setminus (P^{dn},Q)$ are monomials in the sets $(x_1) \setminus (x_1y^{d},x_1^{e_1+1})$, $(x_i^{(i-1)d}) \setminus (x_i^{(i-1)d}y^{e_i+1},x_i^{id})$, $i = 2,...,m-1$, and $(x_m^{(m-1)d}) \setminus (x_m^{(m-1)d}y^{e_{m}+1},x_m^{nd}).$
From this it follows that $a(\widetilde{(P^{d},Q)}/(P^{d},Q))$ is attained among the degrees of the elements 
$x_1^{e_1}y^{d-1},x_{i}^{id-1}y^{e_i},$ $i = 2,...,m-1$, and $x_m^{nd-1}y^{e_m}$.
Hence,
$$a(\widetilde{(P^{dn},Q)}/(P^{dn},Q)) = \max\{id + e_i-1|\ i = 1,...,m-1,n\} = nd + e_n-1$$
because $e_m = e_n$ and $id + e_i \le nd+e_n$ for $i = 1,...,n$. Therefore, $\sdeg I^n = nd + e_n$, as required.
\end{proof}

Given any convergent sequence $\{e_n\}$ of non-negative integers, the condition $e_n \le e_{n+1} + d$ is satisfied if $d$ is large enough. Therefore, $\{e_n\}$ is the defect sequence of the function $\sdeg I^n$ for some equigenerated ideal $I$ with $\dim R/I > 0$. 

If $d$ is given, we can only show that the defect sequence of  $\sdeg I^n$ can be any weakly decreasing sequence of non-negative integers.

\begin{Theorem} \label{defect3}
Let $\{e_n\}$ be an arbitrary weakly decreasing sequence of non-negative integers. Then $\{e_n\}$ is the defect sequence of the function $\sdeg I^n$ of an ideal $I$ equigenerated in any degree $d$.
\end{Theorem}

\begin{proof}
We only need to construct an ideal $I$ equigenerated in degree $d$ with $\dim R/I > 0$ such that $\sdeg I^n = nd + e_n$ for all $n \ge 1$.

Let $m$ be the stability index of the sequence $\{e_n\}$. Let $S=k[x,y,z,u]$ and
$$Q =  (x^{e_1+2}, x^{e_1+1}y,...,x^{e_{m-1}+1}y^{m-1},y^m)  + (x,y,z^{e_m+1})u.$$
Set $R = S/Q$ and $I=(y(y,z)^{d-1},u^d)R$. 

Recall that $\sdeg I^n = a(\widetilde{I^n} /I^n) +1$.
Since $z$ is a system of parameters for $R/I$, 
$$\widetilde{I^n} = \cup_{t \ge 1}I^n:z^t = (u,y^n)R+I^n.$$ 
Since $yu \in Q$, 
$\widetilde{I^n} /I^n = (uR+I^n/I^n) \oplus (y^nR+I^n/I^n).$
From this it follows that
\begin{align*}  \tag{1}
a(\widetilde{I^n} /I^n) & = \max\{a(uR+I^n/I^n), a(y^nR+I^n/I^n)\}\\
 & = \max\{a(R/I^n:u)+1, a(R/I^n:y^n) + n\}.
\end{align*} 

We have $R/I^n:u  \cong k[u]/(z^{e_m+1},u^{dn-1}).$
Hence, 
\[ \tag{2}
a(R/I^n:u) = dn + e_m -2.
\] 

On the other hand, we have
\begin{align*}
R/I^n:y^n  & \cong k[x,y,z]/(x^{e_1+2},x^{e_1+1}y,...,x^{e_{m-1}+1}y^{m-1}):y^{n} + (y,z)^{(d-1)n}.
\end{align*}
Thus, $a(R/I^n:y^n)$ is the maximal degree of the monomials in $k[x,y,z]$ not contained in $(x^{e_1+2},x^{e_1+1}y,...,x^{e_{m-1}+1}y^{m-1}):y^{n} + (y,z)^{(d-1)n}$.

If $n < m$, we have 
$$(x^{e_1+1}y,...,x^{e_{m-1}+1}y^{m-1}):y^{n} = (x^{e_n+1},x^{e_{n+1}+1}y,...,x^{e_{m-1}+1}y^{m-n-1}).$$
From this it follows that $a(R/I^n:y^n)$ is attained among the degrees of the monomials
$$x^{e_n}z^{(d-1)n-1},x^{e_{n+1}}yz^{(d-1)n-2},...,x^{e_{m-1}}y^{m-n-1}z^{(d-1)n-m+n}.$$
Hence,
\begin{align*}
a(R/I^n:y^n) & = \max\{e_i + (d-1)n-1|\ i = n,...,m-1\} = e_n + (d-1)n-1
\end{align*}
because $\{e_n\}$ is a weakly decreasing sequence.
Therefore, by (1) and (2) we obtain
$$a(\widetilde{I^n} /I^n) = \max\{dn + e_m - 1, dn + e_n - 1\} = dn + e_n - 1.$$

If $n \ge m$, we have $\widetilde{I^n}= uR+I^n$. Hence
$$a(\widetilde{I^n} /I^n) = a(R/I^n:u) + 1 = dn + e_m -1.$$
\end{proof}

We call the initial part of the function $\sdeg I^n$ before it becomes a linear function its {\em initial sequence}. 
The following consequence of  Theorem \ref{defect2} and  Theorem \ref{defect3} shows that this sequence can be any finite monotone sequence.

\begin{Corollary} \label{monotone}
Any finite monotone sequence of non-negative integers $s_n \ge nd$, $n = 1,...,m$, is the initial sequence of the function $\sdeg I^n$ of an ideal $I$ equigenerated in degree $d$.
\end{Corollary}

\begin{proof}
Set $e_n = s_n - dn$ for $n \le m$ and $e_n = e_m$ for $n > m$. We only need to show that $\{e_n\}$ is the defect sequence of the function $\sdeg I^n$ of an ideal $I$ equigenerated in degree $d$.
If $\{s_n\}$ is a weakly increasing sequence, then $e_n \le e_{n+1}+d$ for all $n \ge 1$.
Hence, the conclusion follows from Theorem \ref{defect2}.
If $\{s_n\}$ is a weakly decreasing sequence, then $e_n \ge e_{n+1}+d$ for $n < m$.
Hence $\{e_n\}$ is a weakly decreasing sequence. 
Therefore, the conclusion follows from Theorem \ref{defect3}.
\end{proof}

Another consequence of Theorem \ref{defect3} is that  the function $\sdeg I^n$ can fluctuate everywhere
before it becomes a linear function. It is easy to construct a function $\sdeg I^n$ with local maxima and minima at any sequence of prescribed places. 

\begin{Example}
Let $1 = m_1 < m_2 < \cdots < m_{2s+1} = m$
be any sequence of positive integers. Let $d \ge 1$ and $e \ge 0$.
Using descending induction we can construct a weakly decreasing sequence $\{e_n\}$ of non-negative integers by the following rules:
\begin{enumerate}
\item $e_n = e$ for $n \ge m$.
\item $e_n = e_{m_{2i+1}} + (d+1)(m_{2i+1}-n)$ if $m_{2i+1}-1 \ge n \ge m_{2i}$, $i = s,...,1$.
\item $e_n = e_{m_{2i}} + (d-1)(m_{2i}-n)$ if $m_{2i}-1 \ge n \ge m_{2i-1}$, $i = s,...,1$.
\end{enumerate}
By Theorem \ref{defect3} there is an ideal $I$ equigenerated in degree $d$ such that
$\sdeg I^n = dn + e_n$ for $n \ge 1$. By (1), $\sdeg I^n = dn+e$ for $n \ge m$.
By (2) and (3) we have 
\begin{align*}
\sdeg I^n & = \sdeg I^{n+1} + 1\ \text{if $m_{2i+1}-1 \ge n \ge m_{2i}$, $i = s,...,1$},\\
\sdeg I^n & = \sdeg I^{n+1} - 1\ \text{if $m_{2i}-1 \ge n \ge m_{2i-1}$, $i = s,...,1$},
\end{align*}
 Therefore, the function $\sdeg I^n$ has local minima at $n = m_1,...,m_{2s+1}$ and local 
 maxima at $n = m_2,...,m_{2s}$. In particular, if $m_1 < m_2 < \cdots < m_{2s+1}$ is the sequence $1,...,m$, the function $\sdeg I^n$ has local minima and maxima at every number $n \le m$.
\end{Example}

\begin{Remark}
We conjecture that any convergent sequence of non-negative integers is 
the defect sequence of the function $\sdeg I^n$ of an ideal $I$ equigenerated in any given degree $d$.
\end{Remark}


\end{document}